\documentclass[a4paper]{amsart}
\usepackage[utf8]{inputenc}
\usepackage[english]{babel}
\usepackage{amsfonts}
\usepackage{amssymb,amsthm}
\usepackage[left=3cm, right=3cm]{geometry}

\usepackage{mathtools}
\mathtoolsset{showonlyrefs}

\usepackage{hyperref}
\usepackage{xcolor}
\hypersetup{
    colorlinks,
    linkcolor={blue!90!black},
    citecolor={green!50!black},
    urlcolor={blue!90!black}
}

\usepackage{enumitem}
\setenumerate[0]{label={(\roman*)}}

\makeatletter
\def\blfootnote{\gdef\@thefnmark{}\@footnotetext}
\makeatother

\newcommand{\Z}{\mathbb Z}

\newcommand{\R}{\mathbb R}
\newcommand{\C}{\mathbb C}

\newcommand{\mc}{\mathcal}
\renewcommand{\phi}{\varphi}
\newcommand{\n}{\nabla}
\newcommand{\pd}{\partial}

\renewcommand{\tilde}{\widetilde}

\renewcommand{\ge}{\geqslant}
\renewcommand{\le}{\leqslant}
\newcommand{\mf}{\mathfrak}
\newcommand{\Sym}{\mathrm{Sym}}

\renewcommand{\bar}[1]{\mkern 0.5mu\overline{\mkern-0.5mu#1\mkern-0.5mu}\mkern 0.5mu}

\theoremstyle{plain}
\newtheorem{theorem}{Theorem}[section]
\newtheorem*{theorem*}{Theorem}
\newtheorem{lemma}[theorem]{Lemma}
\newtheorem{corollary}[theorem]{Corollary}
\newtheorem{proposition}[theorem]{Proposition}

\newtheorem{conjecture}[theorem]{Conjecture}
\newtheorem{problem}[theorem]{Problem}

\theoremstyle{definition}
\newtheorem{remark}[theorem]{Remark}
\newtheorem{definition}[theorem]{Definition}
\newtheorem{example}[theorem]{Example}

\title[HCF on complex homogeneous manifolds]{Hermitian curvature flow \\on complex homogeneous manifolds}
\author{Yury Ustinovskiy}
\address{Princeton Univsersity, Fine Hall}
\email{yuryu@math.princeton.edu}

\begin{document}
\begin{abstract}
In this paper we study a version of the Hermitian curvature flow (HCF) considered by the author in~\cite{us-16}. We focus on complex homogeneous manifolds equipped with \emph{induced metrics}. We prove that this finite-dimensional space of metrics is invariant under the HCF and write down the corresponding ODE on the space of Hermitian forms on the underlying Lie algebra. Using these computations we construct HCF-Einstein metrics on $G$-homogeneous manifolds, where $G$ is a complexification of a compact simple Lie group. We conjecture that under the HCF any induced metric on such a manifold pinches towards the HCF-Einstein metric. For a nilpotent or solvable complex Lie group equipped with an induced metric we investigate the blow-up behavior of the HCF.
\end{abstract}
\maketitle
\date{}


\section{Introduction}\label{sec:intro}

The striking success of the Hamilton's Ricci flow demonstrates that parabolic metric flows are a very powerful tool that helps construct distinguished Riemannian metrics and study geometry/topology of manifolds admitting certain special metrics. The Ricci flow has particularly nice long-time existence, convergence properties on K\"ahler manifolds and we have reasonably good understanding of the singularity formation~\cite{pss-07, so-ti-17, ti-zh-06, tsu-88}. However, on a general non-K\"ahler Hermitian complex manifold $(M,g,J)$ the Ricci flow does not behave nicely, since the Ricci form $Ric(g)$ is not necessary invariant under the operator of the almost complex structure $J$. Hence the evolved metric $g(t)$ on $M$ is might not be Hermitian. To overcome this problem some natural modifications of the Ricci flow were suggested recently in the literature~\cite{gi-11, st-ti-11, to-we-15}. Any non-K\"ahler manifold admits a one-parameter family of Hermitian connections~\cite{ga-97}, among which the Chern and the Bismut connections are of a special interest. Given a general complex Hermitian manifold it is natural to consider curvatures of these Hermitian connections instead of the Ricci form of the Levi-Civita connection to define a Hermitian version of the Ricci flow. 

In 2011 Streets and Tian~\cite{st-ti-11} introduced a family of Hermitian curvature flows (HCFs) on an Hermitian manifold $(M,g,J)$:
\begin{equation}\label{eq:HCF_intro}
\pd_t g=-S+Q(T),
\end{equation}
where $S_{i\bar j}=g^{m\bar n}\Omega_{m\bar n i\bar j}$ is the second Chern-Ricci form of the Chern connection $\Omega$ and $Q$ is an arbitrary symmetric $J$-invariant quadratic term in torsion $T$ of $\n$. For any choice of $Q$ equation~\eqref{eq:HCF_intro} defines a non-linear strongly parabolic equation for the Hermitian metric $g$.

In this paper we study a version of the Hermitian curvature flow which was considered by the author in~\cite{us-16},
\begin{equation}\label{eq:HCF}
\pd_t g_{i\bar j}=-S_{i\bar j}-\frac{1}{2}g^{m\bar n}g^{p\bar s}T_{mp\bar j}T_{\bar n\bar s i}.
\end{equation}
In what follows we will refer to this flow as the HCF. It is proved in~\cite{us-16} that the flow~\eqref{eq:HCF} preserves Griffiths positivity of the initial Hermitian metric (see Section~\ref{sec:hcf} for the precise formulation). In the present paper we focus on the behavior of this flow on a (not necessary compact) complex homogeneous manifold $M=G/H$ acted on by a complex-analytic group $G$ with the isotropy subgroup $H$. There are several reasons why these manifolds are of a special interest for us.
\begin{enumerate}
\item Homogeneous manifolds form a rich family for which one can explicitly compute certain metrics and analyze behavior of a metric flow. For example, in~\cite{gl-pa-10, la-13} authors study the Ricci flow on Lie groups and homogeneous manifolds; see also~\cite{bo-16,fi-ve-15} for the study of a general HCF version, the \emph{pluriclosed flow}, on compact homogeneous surfaces and solvmanifolds. We expect that these computations will shed some light on the geometric nature of the HCF~\eqref{eq:HCF}.
\item Non-symmetric rational homogeneous manifolds $G/P$, where $G$ is a reductive algebraic group and $P$ is its parabolic subgroup are projective manifolds, however, `natural' Hermitian metrics induced by the Killing metric of the compact form of $G$ are non-K\"ahler~\cite{ya-94}. It is interesting to analyze whether our metric flow distinguishes these special metrics.
\item The HCF flow~\eqref{eq:HCF} is known to preserve Griffiths positivity of the initial metric and homogeneous manifolds equipped with an induced metric (see Definition~\ref{def:induced} below) are essentially the only known source of such examples.
\end{enumerate}

For an homogeneous manifold $M=G/H$ we denote by $\mf g$ and $\mf h$ the complex Lie algebras of $G$ and $H$. Let $\mc M(M)$ be the infinite-dimensional space of Hermitian metrics on $M$ modulo isometry. There are two distinguished subspaces of $\mc M(M)$.
\begin{enumerate}
\item $\mc M^{\rm inv}(M)$~--- the space of $G$-invariant Hermitian metrics on $M$. These metrics are in one-to-one correspondence with $\mathrm{Ad} H$-invariant Hermitian metrics on the vector space $\mf g/\mf h$, see~\cite[X.3]{ko-no-63}:
\[
\mc M^{\rm inv}(M)\longleftrightarrow
\begin{matrix}
\left\{\mbox{$\mathrm{Ad} H$-invariant  metrics on $\mf g/\mf h$}\right\}
\end{matrix}
\]
Clearly the set $\mc M^{\rm inv}(M)$ is preserved by any metric flow of the form $\pd_t g=R(g)$, where $R$ is a tensor field, such that its value at $x\in M$ is determined by the germ of $g$ at $x$, provided that the solution to the flow is unique. Hence, any such `natural' metric flow defines an ODE on the finite dimensional space $\mc M^{\rm inv}(M)$.

\medskip
\item $\mc M^{\rm ind}(M)$~--- the space of $\mf g$-induced metrics. The holomorphic tangent bundle of a complex homogeneous manifold $M$ is globally generated by $\mf g\subset \Gamma(M,TM)$, so any Hermitian metric $h$ on the vector space $\mf g$ induces a metric on its quotient $TM$. Clearly $\mathrm{Ad}\,G$-equivalent Hermitian metrics on $\mf g$ yield isometric Hermitian structures on~$M$~so we have a map
\[
\begin{matrix}
\left\{\mbox{Hermitian metrics on $\mf g$}\right\}/\mathrm{Ad}\,G
\end{matrix}
\longrightarrow
\mc M^{\rm ind}(M).
\]
Different classes of $\mathrm{Ad}\,G$-equivalent metrics on $\mf g$ might induce the same metric on $M$ depending on $H$, so this map is not necessary bijective. Note that typically the induced Hermitian metrics are not $G$-homogeneous.
\end{enumerate}

In general, there is no reason for the set $\mc M^{\rm ind}(M)$ to be invariant under a metric flow. However, it turns out that the HCF~\eqref{eq:HCF} preserves $\mc M^{\rm ind}(M)$ (see Theorem~\ref{thm:ODE_reduction}). In some sense this not a very surprising result, since this flow preserves Griffiths non-negativity and the metrics in $\mc M^{\rm ind}(M)$ are the only known Griffiths non-negative metrics on a general complex homogeneous manifold. What is somewhat less expected is that there exists an ODE on $\Sym^{1,1}(\mf g)$~--- the set of Hermitian metrics on $\mf g^*$, which induces the HCF on all $G$-homogeneous manifolds independently of the isotropy subgroup $H$.

This ODE turns out to be very similar to the ODE defined by the zero-order part of the evolution equation for the Riemannian curvature under the Ricci flow. Namely, in the case of the Ricci flow the curvature operator $R\in\Sym^{2}(\mf{so}(n))$ evolves (in the moving frame) according to the equation
\[
\pd_t R=\Delta R+R^2+R^\#
\]
and the relevant ODE is
\[
\pd_t R=R^2+R^\#,
\]
where $\#$ is a quadratic operation on $\Sym^{2}(\mf {so}(n))$ introduced by Hamilton in~\cite{ha-86}.
For the HCF on the set of induced metrics $\mc M^{\rm ind}(M)$ the ODE for $h\in\Sym^{1,1}(\mf g)$ takes form
\[
\pd_t h=h^\#,
\]
where $\#$ a generalization of Hamilton's operation for an arbitrary Lie algebra (Definition~\ref{def:sharp} below). This is a Riccati-type system of equations, as the evolution term $h^\#$ is quadratic in $h$.

In a special case when $G$ is the complexification of a compact simple Lie group we use the ODE for $h\in\Sym^{1,1}(\mf g)$ to construct scale-static solutions to the HCF on any homogeneous manifold $M=G/H$. Such a metric corresponds to the Killing metric of the compact real form of $G$ and can be thought of as the HCF-analogue of the Einstein metric. We explicitly solve the ODE in some simple examples (see Example~\ref{ex:hopf} for the diagonal Hopf surface and Example~\ref{ex:Iwasawa} for the Iwasawa threefold) and formulate conjectural pinching behavior of this differential equation (Conjecture~\ref{conj:pinching}). It seems that at this point we need a better understanding of the algebraic properties of the operation~$\#$ to resolve the conjecture and to understand the behavior of the HCF on $\mc M^{\rm ind}(G/H)$ for a general $G$ .

Finally, in Section~\ref{sec:ODE_algebraic_properties} we study blow-up behavior of ODE $\pd_t h=h^\#$ on an arbitrary complex Lie group. Namely, we analyze how the growth rate of a solution $h(t)$ depends on algebraic properties of~$\mf g$ (see Theorems~\ref{thm:ODE_nilpotent} and~\ref{thm:ODE_solvable}).


\section{Hermitian curvature flow}\label{sec:hcf}
In this section we fix some notations and review the definition of the Hermitian curvature flow.

In what follows for a complex vector space $V$ we denote by $\bar V$ the conjugate vector space. We denote by $\Sym^{1,1}(V)\subset V\otimes\bar V$ the real vector space of Hermitian forms, i.e., the set of complex sesquilinear forms $h$ on $V^*$ such that $h(\xi,\bar{\eta})=\bar{h(\eta,\bar{\xi})}$. For $h_1,h_2\in \Sym^{1,1}(V)$ we say that $h_1\ge h_2$, if for any $\xi\in V^*$ the inequality $h_1(\xi,\bar{\xi})\ge h_2(\xi,\bar{\xi})$ holds. A form $h\in \Sym^{1,1}(V)$ is said to be \emph{positive}, if $h(\xi,\bar{\xi})>0$ for any nonzero $\xi\in V^*$.

Let $(M,g,J)$ be an Hermitian manifold, where $J$ is an integrable complex structure and $g$ is an Hermitian metric. Let $TM\otimes \C=T^{1,0}M\oplus T^{0,1}M$ be the decomposition of the complexified tangent bundle into the $\pm i$ eigenspaces of $J$. Denote by $\n$ the Hermitian connection on $TM$ uniquely characterized by the properties
\begin{enumerate}
\item $\n g=0$,
\item $\n J=0$,
\item $T(X,JY)=T(JX,Y)$ for any $X,Y\in TM$, where $T(X,Y):=\n_X Y-\n_Y X-[X,Y]$ is the torsion tensor,
\item $\n^{0,1}=\bar\pd$.
\end{enumerate}
\begin{remark}
Given (i) and (ii) properties (iii) and (iv) are equivalent. It is useful to think about (iii) as the vanishing of $(1,1)$-type part of $T$, i.e., $T(\xi,\bar{\eta})=0$ for $\xi,\eta\in T^{1,0}M$.
\end{remark}

\begin{definition}
The \emph{Chern curvature} of an Hermitian manifold $(M, g, J)$ is the curvature of $\n$.
\[
\Omega(X,Y)Z:=(\n_X\n_Y-\n_Y\n_X-\n_{[X,Y]})Z,
\]
where $X,Y,Z\in TM$. We also define a tensor with 4 vector arguments by lowering one index.
\[
\Omega(X,Y,Z,W):=g\bigl((\n_X\n_Y-\n_Y\n_X-\n_{[X,Y]})Z, W\bigr).
\]
\end{definition}

Tensor $\Omega$ satisfies a number of symmetries.
\begin{proposition}\label{prop:curvature_symmetries}
For any real vectors $X,Y,Z,W\in TM$ one has
\begin{enumerate}
\item[] $\Omega(X,Y,Z,W)=-\Omega(Y,X,Z,W)$, $\Omega(X,Y,Z,W)=-\Omega(X,Y,W,Z)$;
\item[] $\Omega(JX,JY,Z,W)=\Omega(X,Y,Z,W)$, $\Omega(X,Y,JZ,JW)=\Omega(X,Y,Z,W)$.
\end{enumerate}
\end{proposition}
Symmetries of $\Omega$ imply that for $\xi,\eta,\zeta,\nu\in T^{1,0}M$
\[
\Omega(\xi,\eta,\cdot,\cdot)=\Omega(\cdot,\cdot,\zeta,\nu)=0,\quad
\Omega(\xi,\bar\eta,\zeta,\bar{\nu})=\bar{\Omega(\eta,\bar\xi,\nu,\bar{\zeta})},
\]
in particular $\Omega(\xi,\bar \xi,\eta,\bar\eta)\in\R$. It is easy to check that the values $\Omega(\xi,\bar \xi,\eta,\bar\eta)$ for $\xi,\eta \in T^{1,0}M$ completely determine tensor $\Omega$. In what follows, by abuse of notation we sometimes write $TM$ as a shorthand for the holomorphic tangent bundle $T^{1,0}M$.

The main object of our study is the following version of a general Hermitian curvature flow~\eqref{eq:HCF_intro}
\begin{equation}\label{eq:HCF_2}
\begin{cases}
\pd_t g_{i\bar j} = -\Theta_{i\bar j}(g(t)),\\
g(0)=g_0,
\end{cases}
\end{equation}
where $\Theta(g)\in \Sym^{1,1}(T^*M)$ is given by 
\[
\Theta(g)_{i\bar j}:=g^{m\bar n}\Omega_{m\bar n i\bar j}+
\frac{1}{2}g^{m\bar n}g^{p\bar s}T_{pm\bar j}T_{\bar s\bar n i}
\] 
is the sum of the second Chern-Ricci curvature and a certain quadratic torsion term for $g=g(t)$. We refer to $\Theta(g)$ as the \emph{torsion-twisted Chern-Ricci form}. By~\cite[Prop.\,5.1]{st-ti-11} on a compact manifold there exists unique solution to equation~\eqref{eq:HCF_2} on some maximal time interval $[0,\tau)$. If the initial metric $g_0$ is K\"ahler, the HCF coincides with the K\"ahler-Ricci flow.

Important property of the flow~\eqref{eq:HCF_2} is preservation of the following curvature positivity condition.
\begin{definition}\label{def:positive}
A complex Hermitian manifold $(M,g,J)$ has \emph{Griffiths positive} (resp.\,non-negative) curvature, if its Chern curvature $\Omega$ for any non-zero $\xi,\eta\in T^{1,0}M$ satisfies
\[
\Omega(\xi,\bar \xi, \eta,\bar \eta)>0, (\mbox{resp.} \ge0).
\]
\end{definition}
\begin{theorem}[{\cite[Theorem\,0.1]{us-16}}]
Let $g(t), t\in[0,\tau)$ be the solution to the HCF on a compact complex Hermitian manifold $(M,g_0,J)$. Assume that the Chern curvature $\Omega^{g_0}$ at the initial moment $t=0$ is Griffiths non-negative (resp. positive). Then for $t\in[0,\tau)$ the Chern curvature $\Omega(t)=\Omega^{g(t)}$ remains Griffiths non-negative (resp. positive).

If, moreover, the Chern curvature $\Omega^{g_0}$ is Griffiths positive at least at one point $x\in M$, then for any $t\in(0;\tau)$ the Chern curvature is Griffiths positive everywhere on $M$.
\end{theorem}

One of the main sources of complex manifolds admitting metrics of Griffith non-negative curvature is the set of manifolds with \emph{globally generated} tangent bundle. To be precise, if $\C^m\to TM$ is a surjective map from a trivial bundle over a complex manifold $M$, then any constant metric on $\C^m$ induces a metric with a Griffith non-negative curvature~\cite[\S 0.5]{gr-ha-94}. In the next section we investigate the behavior of HCF on such manifolds, equipped with an induced metric.


\section{HCF on a globally generated tangent bundle}\label{sec:hcf_globally_generated}

We start this section with deriving the coordinate expression for the torsion-twisted Chern-Ricci form
\[
\Theta_{i\bar j}=g^{m\bar n}\Omega_{m\bar n i\bar j}+\frac{1}{2}g^{m\bar n}g^{p\bar s}T_{mp\bar j}T_{\bar n\bar s i},
\]
on $(M, g, J)$. To simplify subsequent applications of this computation we provide a formula for the $g$-dual of $\Theta_{i\bar j}$
\[
\Theta^{i\bar j}:=g^{i\bar n}g^{m\bar j}\Theta_{m\bar n},
\]
i.e., we use the metric $g$ to identify $\Theta$ with a section of $\Sym^{1,1}(TM)$.

\begin{proposition}\label{prop:Theta_coord}
\begin{equation}\label{eq:Theta_coord}
\Theta^{i\bar j}=
g^{m\bar n}\pd_{m}\pd_{\bar n}g^{i\bar j}-\pd_mg^{i\bar n}\pd_{\bar n}g^{m\bar j}.
\end{equation}
\end{proposition}
\begin{proof}
Recall that for the Chern connection on $(TM, g)$ we have:
\[
\Gamma_{ij}^k=g^{k\bar l}\pd_i g_{j\bar l}.
\]
Hence for the Chern curvature we have:
\begin{equation}
\begin{split}
g^{m\bar l}\Omega_{i\bar j m}^{k}&=
-g^{m\bar l}\pd_{\bar j}\bigl(g^{k\bar n}\pd_ig_{m\bar n} \bigr)=
-g^{m\bar l}\pd_{\bar j}\Bigl(
	-g^{k\bar n}g_{m\bar s}g_{p\bar n}\pd_i g^{p\bar s}
\Bigr)=
g^{m\bar l}\pd_{\bar j}\Bigl(
	g_{m\bar s}\pd_i g^{k\bar s}
\Bigr)=\\&=
g^{m\bar l}\pd_{\bar j}g_{m\bar s}\pd_ig^{k\bar s}+\pd_i\pd_{\bar j}g^{k\bar l}=
-g_{p\bar s} \pd_{\bar j}g^{p\bar l}\pd_i g^{k\bar s}+\pd_i\pd_{\bar j}g^{k\bar l}.
\end{split}
\end{equation}
For the torsion tensor $T_{mp}^i$ we compute:
\begin{equation}
\begin{split}
T_{mp}^i=
g^{i\bar l}(\pd_m g_{p\bar l}-\pd_p g_{m\bar l})=
g_{m\bar l}\pd_p g^{i\bar l} - g_{p\bar l}\pd_m g^{i\bar l}.
\end{split}
\end{equation}
Therefore
\begin{equation}
\begin{split}
\frac{1}{2}g^{m\bar n}g^{p\bar s}T_{mp}^iT_{\bar n\bar s}^{\bar j}&=
\frac{1}{2}g^{m\bar n}g^{p\bar s}(
		g_{m\bar l}\pd_p g^{i\bar l} - g_{p\bar l}\pd_m g^{i\bar l}
	)(
		g_{k\bar n}\pd_{\bar s} g^{k\bar j} - g_{k\bar s}\pd_{\bar n} g^{k\bar j}
	)=\\&=
g^{p\bar s}g_{k\bar l}\pd_p g^{i\bar l}\pd_{\bar s}g^{k\bar j}-
\pd_pg^{i\bar n}\pd_{\bar n}g^{p\bar j}.
\end{split}
\end{equation}
Using the above formulas together we get the expression for $\Theta^{i\bar j}$.
\begin{equation}
\begin{split}
\Theta^{i\bar j}=
g^{m\bar n}\pd_{m}\pd_{\bar n}g^{i\bar j}-\pd_pg^{i\bar s}\pd_{\bar s}g^{p\bar j}.
\end{split}
\end{equation}
After relabeling indices we get the stated formula.
\end{proof}

Combining Proposition~\ref{prop:Theta_coord} and the equation of the HCF flow~\eqref{eq:HCF_2} we get the following corollary.

\begin{corollary}\label{cor:HCF_inverse}
Let $(M,J,g_0)$ be an Hermitian manifold. Assume that $\tilde{g}(t)\in \Sym^{1,1}(TM)$ is a solution to the PDE on $M\times [0,\tau)$
\begin{equation}\label{eq:HCF_for_g_inverse}
\begin{cases}
\pd_t\tilde{g}^{i\bar j}=\tilde{g}^{m\bar n}\pd_m\pd_{\bar n} \tilde{g}^{i\bar j}-\pd_m \tilde{g}^{i\bar n}\pd_{\bar n}\tilde{g}^{m\bar j},\\
\tilde{g}(0)=g_0^{-1},
\end{cases}
\end{equation}
such that $\tilde{g}^{i\bar j}(t)$ is positive definite for $t\in [0,\tau)$. Then $g(t):=\tilde{g}^{-1}(t)$ is the solution to the HCF on~$M$.
\end{corollary}
An interesting feature of equation~\eqref{eq:HCF_for_g_inverse} is that its right-hand side depends only on $\tilde{g}$ and not on $\tilde{g}^{-1}$. In particular, there might exist a solution $\tilde{g}(t)$ starting with a degenerate or indefinite form. It would be interesting to find a geometric interpretation of such solution. However, in this case the equation is not elliptic anymore.
\medskip

Now, assume that the holomorphic tangent bundle  $TM$ is globally generated. If $M$ is compact, the space of global sections $\Gamma(M,TM)$ is finite dimensional. If $M$ is not necessarily compact, $\Gamma(M,TM)$ might be infinite dimensional. Nevertheless, global generation implies that there exists a \emph{finite} collection $\{s_\alpha\}_{\alpha=1}^m$ of elements in $\Gamma(M,TM)$ such that at any point $x\in M$ holomorphic vector fields $\{s_\alpha\}$ generate $TM$, i.e., the natural evaluation map
\[
\mathrm{ev}_x\colon \Gamma(M,TM)\to T_xM
\]
restricted to $V:=\mathrm{span}(s_1,\dots,s_m)$ is surjective (see~\cite[Prop.\,11.2]{de}). In this situation we say that a subspace $V\subset \Gamma(M,TM)$ \emph{generates} $TM$.
\begin{definition}\label{def:induced}
An Hermitian metric $g$ on $M$ is said to be \emph{induced} from an Hermitian metric $h$ on a vector space $V\subset \Gamma(M,TM)$ generating $TM$, if it has the form $g=j_*h\,\vline\,_{(\ker \mathrm{ev})^\perp}$, where
\[
\mathrm{ev}_x\colon V\to T_xM
\]
is the evaluation map and
\[
j_x\colon (\ker \mathrm{ev}_x)^\perp\to T_xM
\]
is the natural isomorphism. In this case we write $g=\mathrm{ev}_*h$ for $h\in\Sym^{1,1}(V^*)$. The set induced metrics on $M$ is denoted by $\mc M^{\rm ind}(M)$. Note that $\mc M^{\rm ind}(M)$ depends on the choice of generating space $V\subset \Gamma(M,TM)$.
\end{definition}

Let $s_{\alpha}=a^i_{\alpha}\,\pd/\pd z_i$ be the local coordinate expression for the vector field $s_\alpha$ and denote $a_{\bar{\alpha}}^{\bar i}:=\bar{a^i_{\alpha}}$. Functions $a^i_\alpha$ are holomorphic. In the coordinates the induced metric $g=g_{i\bar j}$ is given by the expression
\begin{equation}\label{eq:g_a_coord}
g_{i\bar j}=\Bigl(a_\alpha^ia^{\bar j}_{\bar{\beta}}h^{\alpha\bar\beta}\Bigr)^{-1}.
\end{equation}
Indeed $a_\alpha^ia^{\bar j}_{\bar{\beta}}h^{\alpha\bar\beta}$ is the Hermitian metric on $T^*M$ induced by the inclusion $\mathrm{ev}_x^*\colon T^*_xM\to V^*$ and the Hermitian metric $g_{i\bar j}$ on $TM$ is its inverse.

\begin{remark}
Let $\mc E\to M$ be a holomorphic globally generated vector bundle. Notions of global generation and Griffiths non-negative curvature extend to this case in a straightforward manner. It is known that metric $h$ on $\mc E$ induced from a metric on the generating vector space $V\subset \Gamma(M,\mc E)$ has Griffiths non-negative Hermitian curvature $R_{\mc E,h}$:
\[
R_{\mc E, h}(v,\bar v, e,\bar e)\ge 0
\]
for any $x\in M$, $v\in T_x^{1,0}M$, $e\in\mc E_x$. In particular in our setting the Chern curvature $\Omega$ is Griffiths non-negative.
\end{remark}

\begin{proposition}\label{prop:hcf_lie_bracket}
Let $(M, g, J)$ be a complex manifold, such that $TM$ is globally generated by $V\subset \Gamma(M,TM)$. Assume that $g={\rm ev}_*h$ for $h\in \mathrm{Sym}^{1,1}(V^*)$. Let $\{s_\alpha\}_{\alpha=1}^m$ be an $h$-orthonormal frame of $V$. Then the torsion-twisted Chern-Ricci form $\Theta(g)$ considered as a section of $\Sym^{1,1}(TM)$ is given by
\begin{equation}
\Theta(g)=\frac{1}{2}\sum_{\alpha,\beta=1}^{m}[s_{\alpha},s_{\beta}]\otimes\bar{[s_{\alpha},s_{\beta}]},
\end{equation}
where $[\cdot,\cdot]$ is the commutator of vector fields on $M$.
\end{proposition}
\begin{proof}
First we note, that since vector fields $s_\alpha=a^i_{\alpha}\pd/\pd z_i$ are holomorphic, all the derivatives $\bar{\pd} a^i_\alpha$, $\pd a^{\bar j}_{\bar \beta}$ vanish. Using this fact and equations~\eqref{eq:Theta_coord} and~\eqref{eq:g_a_coord} we find
\begin{equation}\label{eq:Theta_lie_bracket}
\begin{split}
\Theta(g)^{i\bar j}&=
g^{m\bar n}\pd_{m}\pd_{\bar n}g^{i\bar j}-\pd_mg^{i\bar n}\pd_{\bar n}g^{m\bar j}=\\&=
h^{\gamma\bar\delta}a_\gamma^m a_{\bar\delta}^{\bar n} h^{\alpha\bar\beta}\pd_m a^i_\alpha \pd_{\bar n}a^{\bar j}_\beta-
h^{\gamma\bar\delta}\pd_m a^i_{\gamma} a^{\bar n}_{\bar\delta}h^{\alpha\bar\beta}a_\alpha^m\pd_{\bar n}a^{\bar j}_{\bar\beta}=\\&=
h^{\alpha\bar\beta}h^{\gamma\bar\delta}
	\bigl(
		a_\gamma^m\pd_m a_\alpha^i\cdot a_{\bar\delta}^{\bar n}\pd_{\bar n}a_{\bar\beta}^{\bar j}-
		a_\alpha^m\pd_m a_\gamma^i\cdot a_{\bar\delta}^{\bar n}\pd_{\bar n}a_{\bar\beta}^{\bar j}
	\bigr)=\\&=
\frac{1}{2}h^{\alpha\bar\beta}h^{\gamma\bar\delta}
	\bigl(
		a_\gamma^m\pd_m a_\alpha^i-a_\alpha^m\pd_m a_\gamma^i
	\bigr)\cdot
	\bigl(
		a_{\bar\delta}^{\bar n}\pd_{\bar n}a_{\bar\beta}^{\bar j}-
		a_{\bar\beta}^{\bar n}\pd_{\bar n}a_{\bar\delta}^{\bar j}
	\bigr).
\end{split}
\end{equation}
In the last equality we used the fact that the whole expression is symmetric under the change $(\alpha\bar{\beta})\leftrightarrow(\gamma\bar{\delta})$.
The last two multiples in the last expressions are exactly the $i$ and $\bar j$ coordinates of the Lie brackets $[a_\gamma^m\pd/\pd z_m,a_\alpha^m\pd/\pd z_m]$ and $[a_{\bar\delta}^{\bar n}\pd/\pd z_{\bar n},a_{\bar\beta}^{\bar n}\pd/\pd z_{\bar n}]$. Since $\{s_\alpha\}_{\alpha=1}^m$ is a $h$-orthonormal basis of $\Gamma(M,TM)$, we get the stated formula.
\end{proof}
The expression~\eqref{eq:Theta_lie_bracket} for $\Theta(g)$ suggests that the HCF with $g_0=\mathrm{ev}_*h \in\mc M^{\rm ind}(M)$ is governed by the Lie algebra of holomorphic vector fields $\Gamma(M,TM)$. In the next sections we study this relation.


\section{Operation \# on a Lie algebra}\label{sec:sharp_definition}
Before studying the HCF on complex homogeneous manifolds following Hamilton~\cite{ha-86} we define an algebraic operation on the second tensor power of a Lie algebra and list its basic properties.

\begin{definition}\label{def:sharp}
Let $\mf g_\R$ be a real Lie algebra. Define a symmetric bilinear $\mathrm{ad}\,\mf g_\R$-invariant operation
\[
\#\colon \mf g_\R^{\otimes 2}\otimes \mf g_\R^{\otimes 2}\to \mf g_\R^{\otimes 2},
\]
by the formula
\begin{equation}
(v_1\otimes v_2)\#(w_1\otimes w_2)=[v_1,w_1]\otimes [v_2,w_2].
\end{equation}
If we choose a basis $\{e_\alpha\}$ of $\mf g_\R$ and denote by $c_{\alpha\beta}^\gamma$ its structure constants, then for $h=\{h^{\alpha\beta}\}$, $k=\{k^{\alpha\beta}\}$, $h,k\in\mf g_\R^{\otimes 2}$
\[
(h\#k)^{\alpha\beta}=
c_{\epsilon\delta}^\alpha c_{\gamma\theta}^{\beta}
h^{\epsilon\gamma}k^{\delta\theta}.
\]

Clearly, the $\#$-operation preserves the parity of the decomposition $\mf g_\R^{\otimes 2}=\Sym^2(\mf g_\R)\oplus \Lambda^2(\mf g_\R)$, i.e., defines the maps
\begin{equation}
\begin{split}
\#&\colon \Sym^2(\mf g_\R)\otimes \Sym^2(\mf g_\R)\to \Sym^2(\mf g_\R),\\
\#&\colon \mathrm{\Lambda}^2(\mf g_\R)\otimes \mathrm{\Lambda}^2(\mf g_\R)\to \Sym^2(\mf g_\R),\\
\#&\colon \mathrm{\Lambda}^2(\mf g_\R)\otimes \Sym^2(\mf g_\R)\to \mathrm{\Lambda}^2(\mf g_\R).
\end{split}
\end{equation}
Now, let $\mf g$ be a complex Lie algebra. Similarly to the real case we denote by the symbol $\#$ the map
\[
\#\colon \mf (\mf g\otimes\bar{\mf g})\otimes (\mf g\otimes\bar{\mf g})\to (\mf g\otimes\bar{\mf g}),
\]
\[
(v_1\otimes\bar v_2)\# (w_1\otimes \bar w_2)=[v_1,w_1]\otimes [\bar v_2,\bar w_2].
\]
As in the real case, $\#$ preserves parity of $\mf g\otimes \bar{\mf g}$, in particular, induces a bilinear map on the set of Hermitian elements of $\mf g\otimes\bar{\mf g}$
\[
\#\colon \Sym^{1,1}(\mf g)\otimes \Sym^{1,1}(\mf g)\to \Sym^{1,1}(\mf g).
\]
We will write $h^{\#}:=\frac{1}{2}h\#h$ for the square of the $\#$-operation.
\end{definition}

\begin{remark}
The $\#$-operation was introduced by Hamilton in the context of the Ricci flow, see~\cite{ha-86}. Definition~\ref{def:sharp} differs from the one of Hamilton in several aspects:
\begin{enumerate}
\item originally $\#$ was defined only for the real Lie algebra $\mf{so}(n)$, while our definition makes sense for any Lie algebra;
\item Hamilton used $\#$ only for the part of $\mathrm{Sym^2}(\mf{so}(n))$ satisfying the first Bianchi identity;
\item Hamilton used the Killing metric to interpret $\#$ as a bilinear operator on the space self-adjoint operations.
\end{enumerate}
This operation and its algebraic properties play the key role in the characterization of compact manifolds with 2-positive curvature operator~\cite{bo-wi-08}. For an arbitrary real metric Lie algebra this operation was also considered by Wilking in~\cite[\S 3]{wi-13}.
\end{remark}

\begin{remark}\label{rk:natural}
Operation $\#$ is natural, i.e., if $\rho\colon \mf g\to \mf h$ is a homomorphism of Lie algebras, then $\rho(h)\#\rho(k)=\rho(h\#k)$ for any $h,k\in\Sym^{1,1}(\mf g)$.
\end{remark}

The following lemma easily follows from the definition by considering a basis of $\mf g$ which diagonalizes the two forms.

\begin{lemma}\label{lemma:sharp_positive}
Let $\mf g$ be a real (resp. complex) Lie algebra. Assume that forms $h,k\in \Sym^2(\mf g)$ (resp. $h,k\in \Sym^{1,1}(\mf g)$) are symmetric (resp. Hermitian) positive definite. Then $h\# k$ is positive semidefinite with $\ker(h\# k)=\mathrm{Ann}([\mf g,\mf g])\subset \mf g^*$.
\end{lemma}

\begin{example}\label{ex:sharp_so3}
Let $\mf g_\R=\mf{su}(2)$ and denote by $\langle\ ,\ \rangle$ the invariant metric on $\mf g_\R$ normalized in such a way that $\langle[e_1,e_2],e_3\rangle=\pm 1$ for any orthonormal triple $e_1,e_2,e_3$. Take $h_\R\in \Sym^2(\mf g_\R)$ and chose a $\langle\ ,\ \rangle$-orthonormal basis $e_1,e_2,e_3$, which diagonalizes $h_\R$ with eigenvalues $\lambda_1,\lambda_2,\lambda_3$. Then 
\begin{equation}
\begin{split}
h_\R^\#&=
1/2(\lambda_1e_1\otimes e_1+\lambda_2e_2\otimes e_2+\lambda_3e_3\otimes e_3)\#
(\lambda_1e_1\otimes e_1+\lambda_2e_2\otimes e_2+\lambda_3e_3\otimes e_3)\\&=
(\lambda_2\lambda_3e_1\otimes e_1+\lambda_1\lambda_3e_2\otimes e_2+\lambda_1\lambda_2e_3\otimes e_3)
\end{split}
\end{equation}
is diagonalized in the same basis with the eigenvalues $\lambda_2\lambda_3, \lambda_1\lambda_3, \lambda_1\lambda_2$. In particular, if $h_\R$ is proportional to the dual of the metric $\langle\ , \ \rangle$ then so is $h_\R^\#$.
\end{example}

\begin{example}\label{ex:sharp_simple}
Let $\mf g_\R$ be a compact simple Lie algebra with an invariant metric $\langle\ ,\ \rangle$. Let $h_\R=\langle\ ,\ \rangle^{-1}$ be the dual of $\langle\ ,\ \rangle$, i.e., for an orthonormal basis $e_1,\dots,e_m$ let $h_\R=\sum e_i\otimes e_i$. Then $h_\R^\#$ is proportional to $h_\R$ with a positive factor.

Indeed, since $\mf g_\R$ is simple, $[\mf g_\R, \mf g_\R]=\mf g_\R$ and by Lemma~\ref{lemma:sharp_positive} both $h_\R$ and $h_\R^\#$ are positive definite $\mathrm{ad}\,\mf g_\R$-invariant elements in $\Sym^2(\mf g_\R)$. As $\mf g_\R$ is simple, such an element is unique up to multiplication by a positive constant, so $h_\R^\#=\lambda h_\R$, $\lambda>0$.
\end{example}

We expect that proportionality $h^\#=\lambda h$ characterizes the $\mathrm{ad}\,\mf g_\R$-invariant positive definite forms on any compact simple Lie algebra.
\begin{problem}
Let $\mf g_\R$ be a compact simple real Lie algebra with an invariant metric $\langle\ ,\ \rangle$. Let $h\in \Sym^2(\mf g_\R)$ be a positive definite form such that $h^\#=\lambda h$. Is $h$ proportional to $\langle\ ,\ \rangle^{-1}$?
\end{problem}

\begin{remark}
If $\mf g_\R$ is a real Lie algebra equipped with an invariant metric $\langle\ ,\ \rangle$ extended in an obvious way to all tensor products of $\mf g_\R$, then a trilinear form
\[
P(h,k,l):=\langle h\#k, l \rangle,\quad h,k,l\in {\mf g_\R}^{\otimes 2}.
\]
is symmetric. It follows from the coordinate expression for $h\# k$ through the structure constants and the fact that in any $\langle\ ,\ \rangle$-orthonormal basis the structure constants are totally skew-symmetric.
\end{remark}

\begin{remark}\label{rk:lie_real_2_complex}
If $\mf g=\mf g_\R\otimes\C$ be the complexification of a real Lie algebra, then a real symmetric form $h_\R\in \Sym^2(\mf g_\R)$ defines an Hermitian form $\phi(h_\R)\in \Sym^{1,1}(\mf g)$: in the basis $\{e_\alpha\}$ of $\mf g_\R$ and the corresponding basis of $\mf g$ this form is given by $\phi(h_\R)_{\alpha\bar{\beta}}:=(h_\R)_{\alpha\beta}$.  It is clear that $\phi(h_\R)\#\phi(h_\R)=\phi(h_\R\#h_\R)$.
\end{remark}


\section{ODE for the induced metrics on complex homogeneous manifolds}\label{sec:hcf_homogeneous}
In this section we assume that $M=G/H$ is a complex homogeneous space with an effective action of a complex Lie group $G$. In this case the Lie algebra $\mf g$ of $G$ naturally embeds into $\Gamma(M,TM)$ and generates $TM$. For $x\in M$ we denote by
\[
\mathrm{ev}_x\colon \mf g\to T_xM,\quad i_x\colon T^*_xM\to\mf g^*
\]
the surjective evaluation map, and the dual embedding. If $M$ is compact and $G=\mathrm{Aut^0}(M)$ then the Lie algebra of $G$ equals $\Gamma(M,TM)$ (see~\cite[\S 2.3]{akh-95}).
In any case the Lie bracket of $\mf g$ coincides with the Lie bracket of vector fields on~$M$.

Let $g_0\in\mc M^{\rm ind}(M)$ be a metric on $M=G/H$ induced from an Hermitian metric $h^{-1}$ on $\mf g$, where $h\in\Sym^{1,1}(\mf g)$. Note that if a group $K$ acts on $M=G/H$ and actions of $G$ and $K$ commute, then the metric $g_0$ is $K$-invariant. We will use this observation Example~\ref{ex:hopf} below.

With Definition~\ref{def:sharp} and Proposition~\ref{prop:hcf_lie_bracket} we can reduce the HCF on a complex homogeneous manifold $(M,g_0,J)$ to an ODE for $h\in \Sym^{1,1}(\mf g)$.

\begin{theorem}\label{thm:ODE_reduction}
Let $M=G/H$ be a complex homogeneous manifold equipped with an induced Hermitian metric $g_0=\mathrm{ev}_*(h_0^{-1})\in \mc M^{\rm ind}(M)$, where $h_0\in \mathrm Sym^{1,1}(\mf g)$.
Let $h(t)$ be the solution to the ODE
\begin{equation}\label{eq:sharp_ODE}
\begin{cases}
\pd_t h=h^\#,\\
h(0)=h_0.
\end{cases}
\end{equation}
Then $g(t)=\mathrm{ev}_*(h(t)^{-1})$ solves the HCF on $(M,g_0,J)$. In particular $g(t)\in\mc M^{\rm ind}(M)$.
\end{theorem}
\begin{proof}
If $h(t)$ solves~\eqref{eq:sharp_ODE}, then by Proposition~\ref{prop:hcf_lie_bracket} the Hermitian metric $\tilde{g}(t)=i^*h(t)$ on $T^*M$ satisfies the partial differential equation
\[
\begin{cases}
\pd_t \tilde{g}=\Theta(\tilde{g}^{-1}),\\
\tilde{g}(0)=g_0^{-1},
\end{cases}
\]
where, as in Proposition~\ref{prop:Theta_coord}, $\Theta$ is identified with a section of $\Sym^{1,1}(TM)$.
Hence $g(t)=\tilde{g}(t)^{-1}$ is the solution to the HCF on $(M,g_0,J)$.
\end{proof}

Surprising consequence of this theorem is that ODE~\eqref{eq:sharp_ODE}
gives solutions to the HCF on all $G$-homogeneous manifolds $M=G/H$	 equipped with an induced metric independently of the isotropy subgroup~$H$.

\begin{example}\label{ex:ode_3d}
Let $G=SL(2,\C)$. Lie algebra of $G$ has the compact real form $\mf{su}(2)$, i.e., $\mf{sl}(2,\C)=\mf{su}(2)\otimes\C$. Assume that $h_0\in \Sym^{1,1}(\mf {sl}(2,\C))$ corresponds to $h_\R\in \Sym^{1,1}(\mf{su}(2))$ (see Remark~\ref{rk:lie_real_2_complex}). Then ODE~\eqref{eq:sharp_ODE} reduces to the equation for $h_\R\in \Sym^2(\mf{su}(2))$
\[
\pd_t h_\R =h_\R^\#.
\]
Let $\langle\ ,\ \rangle$ be a positive definite multiple of the Killing form of $\mf{su}(2)$. Assume that $\langle\ ,\ \rangle$ is normalized in such way that for any orthonormal basis $e_1,e_2,e_3$ we have $\langle[e_1,e_2],e_3\rangle=\pm 1$.

Let $e_1,e_2,e_3$ be an orthonormal basis of $\mf{su}(2)$ diagonalizing $h_\R$. Denote the eigenvalues of $h_\R$ with respect to $\langle\ ,\ \rangle$ by $\lambda_1,\lambda_2,\lambda_3$. In this basis the evolution equation takes the form
\begin{equation}\label{eq:example_ode_3d}
\begin{cases}
\pd_t \lambda_1=\lambda_2\lambda_3\\
\pd_t \lambda_2=\lambda_1\lambda_3\\
\pd_t \lambda_3=\lambda_1\lambda_2.
\end{cases}
\end{equation}
These equations imply that $\det{h_\R}=\lambda_1\lambda_2\lambda_3$ satisfies
\[
\pd_t\det{h_\R}=(\lambda_2\lambda_3)^2+(\lambda_1\lambda_3)^2+(\lambda_1\lambda_2)^2\ge3(\det h_\R)^{4/3}.
\]
So $\det{h_\R}(t)\ge 1/(C-t)^3$ for some $C>0$ and the solution of~\eqref{eq:example_ode_3d} blows up as $t\to t_{\max}<\infty$. Moreover, for any $i,j\in\{1,2,3\}$
\[
\pd_t(\lambda_i^2-\lambda_j^2)=0,
\]
hence all $\lambda_i\to +\infty$ as $t\to t_{\max}$, $i\in\{1,2,3\}$ and $\lambda_i/\lambda_j\to 1$. It follows that $h_\R(t)$ pinches towards the (dual of the) Killing form:
\[
h_\R(t)/|h_\R(t)|_\infty\to \langle\ ,\ \rangle^{-1}=
e_1\otimes e_1+e_2\otimes e_2+e_3\otimes e_3.
\]
\end{example}

\begin{example}[Diagonal Hopf surface]\label{ex:hopf}
Diagonal Hopf surface is the quotient $M=\bigl(\C^2\backslash (0,0)\bigr)/\Gamma$, where the generator of $\Gamma\simeq\Z$ acts as $(z_1,z_2)\mapsto(\lambda z_1,\lambda z_2)$ for some $\lambda\in\C$ with $|\lambda|>1$. $M$ is a compact complex manifold diffeomorphic to $S^3\times S^1$. Note that the natural action of $SL(2,\C)$ on $\C^2$ commutes with $\Gamma$, hence descends to the transitive action on $M$.

As in Example~\ref{ex:ode_3d} any element $h_\R\in\Sym^{2}(\mf{su}(2))$ defines a metric $g_0$ on $M$. By the computation of Example~\ref{ex:ode_3d} and Theorem~\ref{thm:ODE_reduction} under the HCF this metric converges after normalization to the metric $g_{\infty}=\mathrm{ev_*}(\langle\ ,\ \rangle)$. Since $\langle\ ,\ \rangle$ is $SU(2)$-invariant, the limiting metric $g_\infty$ is also $SU(2)$-invariant. Moreover, since the action of $SL(2,\C)$ on $\C^2$ commutes with the scalings, the metric $g_\infty$ pulled back to $\C^2$ is also scaling-invariant. Such a metric is unique up to multiplication by a positive constant $\lambda$:
\[
g_\infty=\lambda\frac{dz_1\otimes d\bar z_1+dz_2\otimes d\bar z_2}{|z_1|^2+|z_2|^2},
\]
where $z_1,z_2$ are coordinates in $\C^2$. This is the so-called \emph{round} metric on a Hopf surface.
\end{example}

Example~\ref{ex:ode_3d} demonstrates the expected behavior of the ODE~\eqref{eq:sharp_ODE} for any complex Lie group $G$ with a simple compact real form. Namely, assume that $\mf g=\mf g_\R\otimes \C$ is the complexification of a simple compact real Lie algebra $\mf g_\R$ with an invariant metric $\langle\ ,\ \rangle$. Let us denote by
\[
\kappa\in\Sym^{1,1}(\mf g)
\] the element corresponding to $\langle\ ,\ \rangle^{-1}\in \Sym^{2}(\mf g_\R)$ (as in Remark~\ref{rk:lie_real_2_complex}). We refer to $\kappa$ as the Killing form. For such $G$ and $\kappa$ we propose the following conjecture.

\begin{conjecture}\label{conj:pinching}
Let $h(t)$ be the solution to the ODE~\eqref{eq:sharp_ODE} on the maximal time interval $[0,t_{\max})$. Then there exists $g\in G$, $\lambda\in\R$ such that $h(t)$ pinches towards $\lambda\mathrm{Ad}_g(\kappa)$:
\[
h(t)/|h(t)|_{\infty}\to \lambda\mathrm{Ad}_g(\kappa),\ t\to t_{\max}.
\]
\end{conjecture}

Let complex Lie group $G$ and $\kappa\in\Sym^{1,1}(\mf g)$ be as above. The following result demonstrates the metric induced by $\kappa$ on a $G$-homogeneous manifold are HCF-Einstein, i.e., scale-static under the flow. This observation provides some evidence for Conjecture~\ref{conj:pinching} to be true.
\begin{theorem}\label{thm:scale-static}
Let $G$ be the complexification of a simple compact Lie group. Let $M=G/H$ be a complex homogeneous manifold equipped with the Hermitian metric $g_0=\mathrm{ev}_*(\kappa^{-1})$. Then $g_0$ is scale-static under the HCF, i.e., $\Theta(g_0)=\lambda g_0$ for some positive constant~$\lambda$.
\end{theorem}

\begin{proof}
In Example~\ref{ex:sharp_simple} we observed that for $h_\R=\langle\ ,\ \rangle^{-1}$
\[
h_\R^\#=\lambda h_\R.
\]
Hence for $\kappa=\phi(h_\R)$ (see Remark~\ref{rk:lie_real_2_complex}) we have $\kappa^\#=\lambda\kappa$. This fact together with Theorem~\ref{thm:ODE_reduction} imply that $g_0$ is scale static under the HCF.
\end{proof}


\section{Growth rate of the HCF solutions on Lie groups}\label{sec:ODE_algebraic_properties}
In this section we study the HCF on a complex Lie group $G$, equipped with an induced metric $g_0\in \mc M^{\rm ind}(G)$ ($g_0$ is identified with its restriction to $\mf g\simeq T_{\rm id} G$). By Theorem~\ref{thm:ODE_reduction} the HCF reduces to the ODE for $h(t)\in \Sym^{1,1}(\mf g)$ with $h_0=g_0^{-1}$
\begin{equation}\label{eq:ODE_sec_algebraic}
\begin{cases}
\pd_t h=h^\#\\
h(0)=h_0.
\end{cases}
\end{equation}
It turns out, that the growth rate of a solution $h(t)$ is completely determined by the algebraic properties of the underlying Lie algebra. Namely, $h(t)$ has polynomial, exponential growth or finite time blow-up, depending on whether $\mf g$ is nilpotent, solvable or admits a semisimple quotient.

Before we state and prove results on the growth rate of a solution to~\eqref{eq:ODE_sec_algebraic} let us make the following elementary observation.
\begin{proposition}\label{prop:ODE_bound}
Let $h(t),k(t)\in\Sym^{1,1}(\mf g)$ be the solutions to~\eqref{eq:ODE_sec_algebraic} with the initial conditions $h(0)=h_0$ and $k(0)=k_0$ such that $h_0\ge k_0\ge 0$. Then for $t\ge 0$
\[
h(t)\ge k(t).
\]
\end{proposition} 
\begin{proof}
We claim that for $h,k\in\Sym^{1,1}(\mf g)$ if $h\ge k\ge 0$, then $h^\#\ge k^\#$. Indeed, let $\{e_i\}$ be a basis diagonalizing simultaneously $h$ and $k$:
\[
h=\sum a_i e_i\otimes\bar e_i,\quad k=\sum b_i e_i\otimes\bar e_i
\]
with $a_i\ge b_i\ge 0$. Then $h^\#=\sum_{i,j}a_ia_j[e_i,e_j]\otimes \bar{[e_i,e_j]}\ge \sum_{i,j}b_ib_j[e_i,e_j]\otimes \bar{[e_i,e_j]}=k^\#$. 

The proposition now follows from, e.g.,~\cite[Lemma~4.1]{ha-86}, as $h(0)-k(0)$ is non-negative and $\pd_t(h(t)-k(t))=h(t)^\#-k(t)^\#$.
\end{proof}

\begin{theorem}\label{thm:ODE_nilpotent}
For a complex Lie algebra $\mf g$ and a positive definite Hermitian form $h_0\in \Sym^{1,1}(\mf g)$ let $h(t)$ be the solution to the ODE~\eqref{eq:ODE_sec_algebraic}
on the maximal time interval $[0;t_\mathrm{max})$, $0<t_\mathrm{max}\le+\infty$. Then the following are equivalent:
\begin{enumerate}
\item[1.] $\mf g$ is a nilpotent Lie algebra;
\item[2.] for any initial condition $h_0$ the solution $h(t)$ has at most \emph{polynomial growth}, i.e., $t_{\max}=+\infty$ and there exists a polynomial $p$ such that
\[
h(t)<p(t) h_0;
\]
\item[3.] for some initial condition $h_0$ the solution $h(t)$ has \emph{subexponential growth}, i.e., $t_{\max}=+\infty$ and for any $\epsilon>0$ there exists $T_\epsilon>0$ such that for $t>T_\epsilon$
\[
h(t)<e^{\epsilon t} h_0.
\]
\end{enumerate}
\end{theorem}
\begin{proof}
We prove implications $1\Rightarrow 2\Rightarrow 3\Rightarrow 1$.
\medskip

$1\Rightarrow 2$. By Ado's theorem for nilpotent Lie algebras~\cite{hoch-66} there exits a faithful representation of $\mf g$ into some $\mf {gl}(V)$ such that $\mf g$ acts by nilpotent endomorphisms. With the use of basic theory of Lie algebras~\cite[\S 3.3]{humph-73} one  can assume that the image of this representation lies in $\mf n(n)$~--- the Lie algebra of strictly upper-triangular $n\times n$ matrices:
\[
\rho\colon \mf g\to \mf n(n).
\]
We extend $\rho$ to a map $\rho\colon  \Sym^{1,1}(\mf g)\to  \Sym^{1,1}(\mf n(n))$ in the obvious way.

Let $\{E_{i,j}|1\le i<j\le n\}$ be the elementary matrices spanning $\mf n(n)$. We fix a collection of positive real numbers $\{f_0^{(k)}\}_{k=1}^{n-1}$ such that the Hermitian form $f_0\in \Sym^{1,1}(\mf n(n))$
\[
{f}_0:=\sum_{1\le i<j\le n}f_0^{(j-i)}E_{i,j}\otimes \bar{E_{i,j}}
\]
is greater than $\rho(h_0)$. Consider the solution $f(t)\in \Sym^{1,1}(\mf n(n))$ to the ODE 
\[
\pd_t f=f^\#
\]
with the initial condition $f(0)=f_0$. After expanding the definition of $f^\#$ we see that this ODE is equivalent to a system of $n-1$ scalar equations 
\[
\pd_t f^{(k)}=\frac{1}{2}\sum_{j=1}^{k-1}f^{(j)} f^{(k-j)}, \quad k=1,\dots,n-1.
\]
Solving these equations inductively for $k=1,\dots,n-1$ we get $f^{(k)}(t)=p_{k-1}(t)$, where $p_{k-1}(t)$ is a polynomial of degree $(k-1)$. 

Hermitian forms $\rho(h(t))$ and $f(t)$ satisfy the same ODE with the initial conditions $\rho(h_0)<f_0$. Therefore Proposition~\ref{prop:ODE_bound} implies that $\rho(h(t))\le f(t)$. Since $f(t)$ has polynomial growth, we get
\[
\rho(h(t))<p(t)f_0
\]
for some polynomial $p(t)$. Finally, using the fact that $\rho$ is faithful and $h_0\in\Sym^{1,1}(\mf g)$ is positive definite, we find a constant $C$ such that
\[
h(t)<Cp(t)h_0.
\]
As $h(t)$ is bounded on any interval $[0,\tau)$, the solution extends to the whole $[0;+\infty)$.
\medskip

$2\Rightarrow 3$. Is trivially true.

$3\Rightarrow 1$. Assume that 1 does not hold and $\mf g$ is not nilpotent. Then by Engel's theorem for some $x\in\mf g$ the operator $\mathrm{ad}_x$ is not nilpotent. Hence the map $\mathrm{ad}_x\colon\mf g\to\mf g$ has non-zero eigenvalue $\lambda$:
\[
[x,y]=\lambda y, y\neq 0.
\]
Consider $f_0:=a_0x\otimes\bar x+b_0y\otimes \bar y\in\Sym^{1,1}(\mf g)$. Note that $f_0^\#=|\lambda|^2a_0b_0 y\otimes \bar y$, hence for the functions $a(t), b(t)$ with $a(0)=a_0$, $b(0)=b_0$ satisfying
\[
\pd_t a=0,\quad \pd_t b=|\lambda|^2 ab
\]
the form $f(t)=a(t)x\otimes\bar x+b(t)y\otimes \bar y$ solves the ODE~\eqref{eq:ODE_sec_algebraic}. Explicitly these functions are given by $a(t)=a_0$, $b(t)=b_0e^{|\lambda|^2a_0t}$. If positive numbers $a_0,b_0$ are small enough, one has $f_0<h_0$, hence by Proposition~\ref{prop:ODE_bound} $f(t)<h(t)$. Therefore $h(t)$ can not have subexponential growth. Contradiction.
\end{proof}

\begin{theorem}\label{thm:ODE_solvable}
For a complex Lie algebra $\mf g$ and a positive definite Hermitian form $h_0\in \Sym^{1,1}(\mf g)$ let $h(t)$ be the solution to the ODE~\eqref{eq:ODE_sec_algebraic}
on the maximal time interval $[0;t_\mathrm{max})$, $0<t_\mathrm{max}\le+\infty$. Then the following are equivalent:
\begin{enumerate}
\item[1.] $\mf g$ is a solvable Lie algebra;
\item[2.] for any initial condition $h_0$ the solution $h(t)$ has at most \emph{exponential growth}, i.e., $t_{\max}=+\infty$ and there exist constants $C,K$ such that
\[
h(t)<C e^{K t} h_0;
\]
\item[3.] for some initial condition $h_0$ the solution $h(t)$ exists on $[0,+\infty)$, i.e., $t_{\max}=+\infty$.
\end{enumerate}
\end{theorem}

\begin{proof}
The proof is essentially analogous to Theorem~\ref{thm:ODE_nilpotent}. We prove implications $1\Rightarrow 2\Rightarrow 3\Rightarrow 1$.
\medskip

$1\Rightarrow 2$. By Ado's theorem there exits a faithful representation of $\mf g$ into some $\mf {gl}(V)$, and by Lie's theorem one  can assume that the image of this representation lies in $\mf b(n)$~--- the Borel subalgebra of $\mf {gl}(V)$, consisting of upper-triangular $n\times n$ matrices:
\[
\rho\colon \mf g\to \mf b(n).
\]

Let $\{E_{i,j}|1\le i\le j\le n\}$ be the elementary matrices spanning $\mf b(n)$. We fix a collection of positive real numbers $\{f_0^{(k)}\}_{k=0}^{n-1}$ such that the Hermitian form $f_0\in \Sym^{1,1}(\mf b(n))$
\[
{f}_0:=\sum_{1\le i\le j\le n}f_0^{(j-i)}E_{i,j}\otimes \bar{E_{i,j}}
\]
is greater than $\rho(h_0)$. Consider the solution $f(t)\in \Sym^{1,1}(\mf b(n))$ to the ODE~\eqref{eq:ODE_sec_algebraic}
with the initial condition $f(0)=f_0$. After expanding the definition of $f^\#$ we see that this ODE is equivalent to a system of $n-1$ scalar equations
\[
\pd_t f^{(k)}=\frac{1}{2}\sum_{j=0}^{k}f^{(j)} f^{(k-j)}, \quad k=1,\dots,n-1
\]
with $f^{(0)}(t)\equiv f^{(0)}_0$. Solving these equations inductively for $k=1,\dots,n-1$ we get $f^{(k)}(t)=q_{k}(e^{f^{(0)}_0 t})$, where $q_{k}(t)$ is a polynomial of degree $k$ with $q_{k}(0)=0$. 

Hermitian forms $\rho(h(t))$ and $f(t)$ solve the same ODE with the initial conditions $\rho(h_0)<f_0$. Therefore Proposition~\ref{prop:ODE_bound} implies that $\rho(h(t))\le f(t)$. Since $f(t)$ has exponential growth, we get
\[
\rho(h(t))<C_0e^{Kt}f_0
\]
for some constants $C_0,K$. Finally, using the fact that $\rho$ is faithful and $h_0\in\Sym^{1,1}(\mf g)$ is positive definite, we find a constant $C$ such that
\[
h(t)<C e^{Kt}h_0.
\]
As $h(t)$ is bounded on any interval $[0,\tau)$, the solution extends to the whole $[0;+\infty)$.
\medskip

$2\Rightarrow 3$. Is trivially true.

$3\Rightarrow 1$. Assume that 1 does not hold and $\mf g$ is not solvable. Denote by $Rad(\mf g)$ the maximal solvable ideal. Then the quotient $\mf g/Rad(\mf g)$ is semisimple Lie algebra and has a simple summand $\mf g_0$. So there is a surjective homomorphism onto a simple Lie algebra.
\[
\rho\colon \mf g\to \mf g_0.
\]
As in the set-up for Conjecture~\ref{conj:pinching} let $\kappa\in \Sym^{1,1}(\mf g_0)$ be a positive-definite Hermitian form corresponding to the Killing metric of the compact real form of $\mf g_0$. Then according to Example~\ref{ex:sharp_simple} and Theorem~\ref{thm:scale-static} $\kappa^\#=\lambda \kappa$ for some $\lambda>0$.

Now, let $h(t)\in\Sym^{1,1}(\mf g)$ be a solution to the ODE~\eqref{eq:ODE_sec_algebraic} defined on $[0,+\infty)$. 
Choose $\epsilon_0>0$ such that $\epsilon_0\kappa<\rho(h(0))$.
If $\epsilon(t)$ satisfies the equation
\[
\pd_t\epsilon=\lambda\epsilon^2,\quad \epsilon(0)=\epsilon_0,
\]
then $f(t)=\epsilon(t)\kappa$ solves the ODE~\eqref{eq:ODE_sec_algebraic} for $f(t)\in\Sym^{1,1}(\mf g_0)$ with the initial data $f(0)=\epsilon_0\kappa$. Explicitly we have
\[
\epsilon(t)=\frac{\epsilon_0}{1-\epsilon_0\lambda t}.
\]
On the one hand we have the solution $f(t)$ to~\eqref{eq:ODE_sec_algebraic} blowing up at the finite time $t=(\epsilon_0\lambda)^{-1}$, on the other hand $f(0)<\rho(h(0))$, hence by Proposition~\ref{prop:ODE_bound} for any $t\ge 0$
\[
f(t)<\rho(h(t)).
\]
Contradiction with the finiteness of $h(t)$ for all $t\in[0,+\infty)$.
\end{proof}

\begin{example}[Iwasawa manifold]\label{ex:Iwasawa}
Let $G$ be the 3-dimensional complex Heisenberg group
\[
G:=\left\{
\left[
\begin{matrix}
1 & a & b \\
0 & 1 & c \\
0 & 0 & 1
\end{matrix}
\right]
\vline
\ a,b,c\in\C
\right\}.
\]
and $\Gamma\subset G$ its discrete subgroup, consisting of matrices with $a,b,c\in\Z[i]$. The quotient $M=G/\Gamma$ is a compact complex manifold called the \emph{Iwasawa manifold}.  $M$ is an example of a compact complex manifold which does not admit any K\"ahler metric. In fact, since $M$ is not formal, there is no complex structure on the underlying differential manifold $M$ admitting a K\"ahler metric.

The Lie algebra of $G$ is $\mf g=\mathrm{span}(\pd_a,\pd_b,\pd_c)\simeq\mf n(3)$. Consider $g_0=\mathrm{ev}_*(h_0^{-1})$, where $h_0\in\Sym^{1,1}(\mf g)$. Denote by $h(t)$ the solution to the ODE~\eqref{eq:ODE_sec_algebraic}. Theorem~\ref{thm:ODE_nilpotent} provides explicit expression for $h(t)$ and, in particular, implies that $h(t)$ polynomially blows up as $t\to\infty$. In fact, since $h^\#$ is proportional to $\pd_b\otimes\bar{\pd_b}$ for any $h\in \Sym^{1,1}(\mf g)$, we see that $\pd_b\otimes\bar{\pd_b}$ is the only coordinate of $h(t)$, which blows up. For the solution $g(t)=\mathrm{ev}_*(h(t)^{-1})$ ot the HCF this means that as $t\to\infty$
\[
g(t)(\pd_b,\pd_b)\to 0,\quad 
g(t)|_{\mathrm{span}(\pd_a,\pd_c)}\equiv
g(0)|_{\mathrm{span}(\pd_a,\pd_c)}.
\]
To get geometric picture consider the Gromov-Hausdorff limit of $(G/\Gamma, g(t))$. It is easy to see that projection onto coordinates $a$ and $c$ defines a holomorphic fibration
\[
\pi\colon G/\Gamma\to\C/\Z[i]\times\C/\Z[i].
\]
The fibers of $\pi$ are the orbits of the flow generated by $\C\!\cdot\!\pd_c$.
The limiting behavior of $g(t)$ implies that in the Gromov-Hausdorff limit the fibers with the induced metric uniformly collapse to a point as $t\to+\infty$ and $G/\Gamma$ collapses to the product of elliptic curves:
\[
(G/\Gamma, g(t))\underset{GH}{\to} (\C/\Z[i]\times\C/\Z[i], g(0)|_{\mathrm{span}(\pd_a,\pd_c)}).
\]
\end{example}
Using the computations of Theorem~\ref{thm:ODE_nilpotent} one can show that the HCF exhibits similar behavior on all complex nilmanifolds of the form $G/\Gamma$, where $G$ is a complex nilpotent group and $\Gamma\subset G$ is a cocompact lattice.


\bibliographystyle{siam}
\bibliography{biblio}

\end{document}